\documentclass[a4paper,12pt,reqno]{amsart}
\usepackage[a4paper,margin=1.5cm,top=2.5cm,bottom=2.5cm,centering,vcentering]{geometry}
\usepackage{hyperref,graphicx,amssymb}
\usepackage{anyfontsize}

\theoremstyle{plain}
\newtheorem{theorem}{Theorem}[section]

\newtheorem{proposition}[theorem]{Proposition}
\newtheorem{cor}[theorem]{Corollary}
\theoremstyle{definition}

\numberwithin{equation}{section}

\numberwithin{equation}{section}

\begin{document}

\title{On the number of partitions of $n$ whose product of the summands is at most $n$}

    \author[P. J. Mahanta]{Pankaj Jyoti Mahanta}
    \address{Gonit Sora, Dhalpur, Assam 784165, India}
    \email{pankaj@gonitsora.com}


\begin{abstract}
We prove an explicit formula to count the partitions of $n$ whose product of the summands is at most $n$. In the process, we also deduce a result to count the multiplicative partitions of $n$.
\end{abstract}

\maketitle

\section{Introduction}
A partition of a non-negative integer $n$ is a representation of
$n$ as a sum of unordered positive integers which are called parts
or summands of that partition. The number of partitions of $n$ is
denoted by $p(n)$. For example, the partitions of 7 are:
\begin{align*}
& 1+1+1+1+1+1+1,\\
& 1+1+1+1+1+2, 1+1+1+1+3, 1+1+1+4, 1+1+5, 1+6, 7,\\
& 1+1+1+2+2, 1+1+2+3, 1+2+4, 2+5, 1+3+3, 3+4,\\
& 1+2+2+2, 2+2+3.
\end{align*}
So, $p(7)=15$.

Many restricted partitions such as partitions with only odd
summands, partitions with distinct summands, partitions whose
summands are divisible by a certain number, partitions with
restricted number of summands, partitions with designated
summands, etc. have been studied over the last two and half
centuries. The partitions of $n$ whose product of the summands is
at most $n$ are another kind of restricted partition. We denote
the total number of such partitions of $n$ by $p_{\leq n}(n)$.
Also, we denote by $p_{=n}(n)$ the number of partitions of $n$
whose product of the summands is equal to $n$. Similarly we use
the notations $p_{<n}(n)$, $p_{\geq n}(n)$, etc.

The value of the product of the summands of a partition depends on
the summands of that partition which are greater than one. We call
these summands the non-one summands or non-one parts. Let $n$ be a
positive integer and the canonical decomposition of $n$ as a
product of distinct primes be
\[n=p_1^{\alpha_1}p_2^{\alpha_2}\cdots p_r^{\alpha_r},\]
where $\alpha_i\in \mathbb{N}$ for all $i=1,2,\dots,r$.

As a consequence of the fundamental theorem of arithmetic, we have the following proposition.
\begin{proposition}\label{pro1}
The product of the summands of a partition of $n$ is equal to $n$
if and only if
\begin{enumerate}
     \item each summand is a divisor of $n$, and
     \item $p_i$ appears as a factor of the summands exactly $\alpha_i$ times, for all $i=1,2,\dots,r$.
\end{enumerate}
\end{proposition}

\begin{proposition}\label{pro2}
Let the product of the summands of a partition of $n$ be $n$, and
each non-one summand of the partition has one prime factor. The
total number of such partitions of $n$ is
\[\prod_{i=1}^rp(\alpha_i).\]
\end{proposition}
\begin{proof}
Let $\lambda_1+\lambda_2+\cdots+\lambda_t$ be a partition of $\alpha_1$. Then,
\[p_1^{\lambda_1}+p_1^{\lambda_2}+\cdots+p_1^{\lambda_t}\leq p_1^{\alpha_1},\]
since, for all positive integers $a,b>1$, we have $a+b\leq ab$ and the equality holds only for $a=b=2$. So, for any positive integer $c>1$, we have $a+b+c\leq ab+c<abc$, and so on.

Here, each non-one summand of a partition of $n$ has only one prime factor. So, by the Proposition \ref{pro1}, we get the result.
\end{proof}

A multiplicative partition of a positive integer $n$ is a
representation of $n$ as a product of unordered positive non-one
integers. The multiplicative partition function was introduced by
MacMahon \cite{macmahon1924dirichlet}, \cite{oppenheim1926arithmetic} in 1923. Since then many
properties of this function have been studied. Some works can be
found in \cite{balasubramanian2011number}, \cite{canfield1983problem}, \cite{chamberland2013multiplicative}, \cite{ghosh2008counting},
\cite{hughes1983number} and \cite{oppenheim1927arithmetic}. $p_{=n}(n)$ is equal to the
total number of multiplicative partitions of $n$. In this paper,
we introduce and prove formulas for $p_{\leq n}(n)$ and
$p_{<n}(n)$, and using both we find $p_{=n}(n)$.

\section{Formulas for $p_{\leq n}(n)$, $p_{<n}(n)$ and $p_{=n}(n)$}

The floor function of $x$, denoted by $\left\lfloor x
\right\rfloor$, is the greatest integer less than or equal to $x$,
where $x$ is any real number. In this section this function
appears many times.

\begin{theorem}\label{th1} We have
\begin{equation}\label{eq:lesseq}
p_{\leq
n}(n)=n+\sum_{k=2}^\ell\sum_{i_1=2}^{\left\lfloor\sqrt[k]{n}\right\rfloor}\sum_{i_2=i_1}^{\left\lfloor\sqrt[k-1]{\frac{n}{i_1}}\right\rfloor}\sum_{i_3=i_2}^{\left\lfloor\sqrt[k-2]{\frac{n}{i_1i_2}}\right\rfloor}\cdots\sum_{i_{k-1}=i_{k-2}}^{\left\lfloor\sqrt{\frac{n}{i_1i_2\cdots
i_{k-2}}}\right\rfloor}\biggl(\left\lfloor\frac{n}{i_1i_2\cdots
i_{k-1}}\right\rfloor-i_{k-1}+1\biggr),
\end{equation}
where $2^\ell\leq n<2^{\ell+1}.$
\end{theorem}

\begin{proof}
$1+1+1+\cdots+1$ is the only partition of $n$ where there is no
non-one summand. The partitions of $n$ where there is only one
non-one summand are
\begin{center}
$1+\cdots+1+1+2,1+\cdots+1+3,\dots,1+1+(n-2),1+(n-1)$ and $n$.
\end{center}
So we get $n-1$ partitions whose product of summands $< n$ and
one partition whose product of summands $=n$.

Now, we count the partitions where there are only two non-one
summands.
\begin{align*}
1+\cdots+1+1+2+2,1+\cdots+1+2+3,1+\cdots+2+4,\dots,1+2+(n-3),2+(n-2),\\
1+\cdots+1+3+3,1+\cdots+3+4,\dots,1+3+(n-4),3+(n-3),\\
\dots \dots,\\
\dots \dots,\\
\text{and so on}.
\end{align*}
Here, in the first row, where the initial non-one summand is 2,
the total number of partitions whose product of summands $\leq n$
is $\left\lfloor\frac{n}{2}\right\rfloor-1$. That of second row is
$\left\lfloor\frac{n}{3}\right\rfloor-2$. In this way we can count
up to the row where initial non-one summand is
$\left\lfloor\sqrt{n}\right\rfloor$; since
$(\left\lfloor\sqrt{n}\right\rfloor+1)(\left\lfloor\sqrt{n}\right\rfloor+1)>n$.
Thus, the total number of such partitions having exactly two
non-one summands is
\[\sum_{i_1=2}^{\left\lfloor\sqrt{n}\right\rfloor}\biggl(\left\lfloor\frac{n}{i_1}\right\rfloor-i_1+1\biggr).\]

To see the clear picture, we now count such partitions where there
are four non-one summands.
{\fontsize{11}{11}\selectfont
\begin{align*}
1+\cdots+1+1+2+2+2+2,1+\cdots+1+2+2+2+3,1+\cdots+2+2+2+4,\dots,2+2+2+(n-6),\\
1+\cdots+1+2+2+3+3,1+\cdots+2+2+3+4,\dots,2+2+3+(n-7),\\
1+\cdots+1+2+2+4+4,\dots,2+2+4+(n-8),\\
\dots \dots,\\
\dots \dots,\\
1+\cdots+1+2+3+3+3,1+\cdots+2+3+3+4,1+\cdots+2+3+3+5,\dots,2+3+3+(n-8),\\
1+\cdots+2+3+4+4,1+\cdots+2+3+4+5,\dots,2+3+4+(n-9),\\
1+\cdots+2+3+5+5,\dots,2+3+5+(n-10),\\
\dots \dots,\\
\dots \dots,\\
1+\cdots+1+3+3+3+3,1+\cdots+3+3+3+4,1+\cdots+3+3+3+5,\dots,3+3+3+(n-9),\\
\dots \dots,\\
\dots \dots,\\
\text{and so on}.
\end{align*}
}

In the first row, the number of partitions whose product of
summands $\leq n$ is $\left\lfloor\frac{n}{2\cdot 2\cdot
2}\right\rfloor-1$. So, in the first group of rows, the number of
partitions whose product of summands $\leq n$ is
\[\sum_{i_3=2}^{\left\lfloor\sqrt{\frac{n}{2\cdot
2}}\right\rfloor}\biggl(\left\lfloor\frac{n}{2\cdot 2\cdot
i_3}\right\rfloor-i_3+1\biggr).\] In the second group of rows it
equals to
\[\sum_{i_3=3}^{\left\lfloor\sqrt{\frac{n}{2\cdot
3}}\right\rfloor}\biggl(\left\lfloor\frac{n}{2\cdot 3\cdot
i_3}\right\rfloor-i_3+1\biggr).\] When the initial non-one summand
is 2, then we can count up to the partitions where the second
non-one summand is
$\left\lfloor\sqrt[3]{\frac{n}{2}}\right\rfloor$. Thus, when the
initial non-one summand is 2, then the number of partitions whose
product of summands $\leq n$ is
\[\sum_{i_2=2}^{\left\lfloor\sqrt[3]{\frac{n}{2}}\right\rfloor}\sum_{i_3=i_2}^{\left\lfloor\sqrt{\frac{n}{2\cdot
i_2}}\right\rfloor}\biggl(\left\lfloor\frac{n}{2\cdot i_2\cdot
i_3}\right\rfloor-i_3+1\biggr).\]
In this way, we can count up to
the group of rows, where the initial non-one summand is
$\left\lfloor\sqrt[4]{n}\right\rfloor$. Therefore, the total
number of such partitions having exactly four non-one summands is
\[\sum_{i_1=2}^{\left\lfloor\sqrt[4]{n}\right\rfloor}\sum_{i_2=i_1}^{\left\lfloor\sqrt[3]{\frac{n}{i_1}}\right\rfloor}\sum_{i_3=i_2}^{\left\lfloor\sqrt{\frac{n}{i_1\cdot
i_2}}\right\rfloor}\biggl(\left\lfloor\frac{n}{i_1\cdot i_2\cdot
i_3}\right\rfloor-i_3+1\biggr).\]

Thus, when there are exactly $k$ non-one summands, then the number
of partitions whose product of summands $\leq n$ is
\begin{align}\label{kExpression}
\sum_{i_1=2}^{\left\lfloor\sqrt[k]{n}\right\rfloor}\sum_{i_2=i_1}^{\left\lfloor\sqrt[k-1]{\frac{n}{i_1}}\right\rfloor}\sum_{i_3=i_2}^{\left\lfloor\sqrt[k-2]{\frac{n}{i_1i_2}}\right\rfloor}\cdots\sum_{i_{k-1}=i_{k-2}}^{\left\lfloor\sqrt{\frac{n}{i_1i_2\cdots
i_{k-2}}}\right\rfloor}\biggl(\left\lfloor\frac{n}{i_1i_2\cdots
i_{k-1}}\right\rfloor-i_{k-1}+1\biggr).
\end{align}

If $2^\ell\leq n<2^{\ell+1}$, then a partition of $n$, whose product of
summands is $\leq n$, must have at most $\ell$ non-one summands. This
completes the proof.
\end{proof}

\begin{cor}\label{cor1} We have
\[p_{\leq n}(n)=p_{<n+1}(n+1).\]
\end{cor}
\begin{proof}
If we observe the expression \ref{kExpression}, then we see that if the product of the
summands is $n$ then $n$ is divisible by $i_1i_2\cdots i_{k-1}$. So, to find $p_{<n}(n)$, we can take
\[\left\lfloor\frac{n-1}{i_1i_2\cdots i_{k-1}}\right\rfloor-i_{k-1}+1,\]
since $a\nmid n$ implies $\left\lfloor\frac{n}{a}\right\rfloor=\left\lfloor\frac{n-1}{a}\right\rfloor$ for all positive integer $a$.

Again, when there are exactly $k$ non-one summands, then the product of the summands is $n$ if $\sqrt[k]{n}$ is an integer. Also, if $\sqrt[k]{n}$ is not an integer, then $\left\lfloor\sqrt[k]{n}\right\rfloor=\left\lfloor\sqrt[k]{n-1}\right\rfloor$. So, to find $p_{<n}(n)$, we can take the value of $i_1$ up to $\left\lfloor\sqrt[k]{n-1}\right\rfloor$. In this way we get, when there are exactly $k$ non-one summands, then the number of partitions whose product of summands $< n$ is
\begin{align*}
	\sum_{i_1=2}^{\left\lfloor\sqrt[k]{n-1}\right\rfloor}\sum_{i_2=i_1}^{\left\lfloor\sqrt[k-1]{\frac{n-1}{i_1}}\right\rfloor}\sum_{i_3=i_2}^{\left\lfloor\sqrt[k-2]{\frac{n-1}{i_1i_2}}\right\rfloor}\cdots\sum_{i_{k-1}=i_{k-2}}^{\left\lfloor\sqrt{\frac{n-1}{i_1i_2\cdots
				i_{k-2}}}\right\rfloor}\biggl(\left\lfloor\frac{n-1}{i_1i_2\cdots
		i_{k-1}}\right\rfloor-i_{k-1}+1\biggr).
\end{align*}
Now, we get the following two cases. 
\begin{enumerate}
	\item $2^\ell< n<2^{\ell+1}$. In this case, $2^\ell\leq n-1<2^{\ell+1}$ and a partition of $n$ whose product of summands is $< n$ must have at most $\ell$ non-one summands.
	\item $n=2^\ell$. In this case, $2^{\ell-1}\leq n-1<2^\ell$ and we must have at most $\ell-1$ non-one summands.
\end{enumerate}
So, combining both cases, we can take the value of $k$ up to $s$, where $2^s\leq n-1<2^{s+1}$. Hence,
\begin{align*}
	p_{<n}(n)=& \  n-1+\sum_{k=2}^{s}\sum_{i_1=2}^{\left\lfloor\sqrt[k]{n-1}\right\rfloor}\sum_{i_2=i_1}^{\left\lfloor\sqrt[k-1]{\frac{n-1}{i_1}}\right\rfloor}\sum_{i_3=i_2}^{\left\lfloor\sqrt[k-2]{\frac{n-1}{i_1i_2}}\right\rfloor}\cdots\sum_{i_{k-1}=i_{k-2}}^{\left\lfloor\sqrt{\frac{n-1}{i_1i_2\cdots
				i_{k-2}}}\right\rfloor}\biggl(\left\lfloor\frac{n-1}{i_1i_2\cdots
		i_{k-1}}\right\rfloor-i_{k-1}+1\biggr)\\
=& \	p_{\leq n-1}(n-1).
\end{align*}

\end{proof}

\begin{cor}\label{cor2} For $n>1$, we have
\[p_{=n}(n)=p_{\leq n}(n)-p_{\leq n-1}(n-1).\]
\end{cor}

\begin{cor} If $n$ is a prime, then
\[p_{\leq n}(n)=p_{\leq n-1}(n-1)+1.\]
\end{cor}

\begin{proof}
If $n$ is a prime, then the product of the summands of any partition of $n$ with more than one non-one summands can not be $n$. So, from the proof of the Corollary \ref{cor1}, we get
\[
\begin{split}
p_{\leq n}(n)
&=1+n-1+\sum_{k=2}^{s}\sum_{i_1=2}^{\left\lfloor\sqrt[k]{n-1}\right\rfloor}\sum_{i_2=i_1}^{\left\lfloor\sqrt[k-1]{\frac{n-1}{i_1}}\right\rfloor}\sum_{i_3=i_2}^{\left\lfloor\sqrt[k-2]{\frac{n-1}{i_1i_2}}\right\rfloor}\cdots\sum_{i_{k-1}=i_{k-2}}^{\left\lfloor\sqrt{\frac{n-1}{i_1i_2\cdots
i_{k-2}}}\right\rfloor}\biggl(\left\lfloor\frac{n-1}{i_1i_2\cdots
i_{k-1}}\right\rfloor-i_{k-1}+1\biggr),\\
&  \hspace{108mm} \text{where} \ 2^s\leq n-1<2^{s+1}\\
&=1+p_{\leq n-1}(n-1).
\end{split}
\]

\end{proof}

\section{Concluding remarks} A few values of $p_{=n}(n)$, $p_{\leq n}(n)$, $p_{\geq n}(n)$ and
$p_{>n}(n)$ can be found in the OEIS sequences
\href{https://oeis.org/A001055}{A001055},
\href{https://oeis.org/A096276}{A096276},
\href{https://oeis.org/A319005}{A319005} and
\href{https://oeis.org/A114324}{A114324} respectively. The Corollary \ref{cor1} implies that the sequence
\href{https://oeis.org/A096276}{A096276} gives the values of
$p_{<n}(n)$ too.

It seems that an explicit formula for $p_{\geq n}(n)$ or
$p_{>n}(n)$ in the spirit of Theorem \ref{th1} can be found. We
leave this as an open problem.

\section*{Acknowledgments}
The author is grateful to Manjil P. Saikia for his helpful
comments on an earlier version of the paper. He also thanks the anonymous referee for useful suggestions.

\bibliographystyle{plain}
\bibliography{ref.bib}

\end{document}